\documentclass[10pt]{amsart}
\usepackage{amsmath,amssymb,amstext,dsfont,fancyvrb,float,fontenc,graphicx,subfigure, amsthm,
latexsym,  comment, graphicx,enumitem}
\usepackage[utf8]{inputenc}
\usepackage[hidelinks]{hyperref}
\hypersetup{%
  citecolor=black
}

\usepackage{indentfirst,csquotes}

\usepackage{thmtools} 

\newcommand{\N}{\mathbb{N}}
\newcommand{\Z}{\mathbb{Z}}

\renewcommand{\epsilon}{\varepsilon}
\renewcommand{\phi}{\varphi}

\usepackage{indentfirst}
\newtheorem{thm}{Theorem}[section]
\newtheorem{lemma}[thm]{Lemma}
\newtheorem{prop}[thm]{Proposition}
\newtheorem{corollary}[thm]{Corollary}

\theoremstyle{definition}
\newtheorem{defn}[thm]{Definition}
\newtheorem{example}[thm]{Example}
\newtheorem{remark}[thm]{Remark}

\hypersetup{ colorlinks=true, linkcolor=black, filecolor=black, urlcolor=black }

\begin{document}
\title{(S,\MakeLowercase{w})-Gap Shifts and Their Entropy} 
\author{Cristian Ramirez and Amy Somers}
\thanks{This work is supported by NSF grant DMS-1954463.}
\date{\today}
\address{Cristian Ramirez, University of California, Berkeley, Berkeley, CA 94720, \emph{E-mail address:} \tt{cramirez@berkeley.edu}}
\address{Amy Somers, University of California, Santa Barbara, Santa Barbara, CA 93117, \emph{E-mail address:} \tt{asomers@ucsb.edu}}
\maketitle

\begin{abstract}
The $S$-gap shifts have a dynamically and combinatorially rich structure. Dynamical properties of the $S$-gap shift can be related to the properties of the set $S$. This interplay is particularly interesting when $S$ is not syndetic such as when $S$ is the set of prime numbers or when $S=\{2^n\}$. It is a well known result that the entropy of the $S$-gap shift is given by $h(X) = \log \lambda$, where $\lambda>0$ is the unique solution to the equation $\sum_{n \in S} \lambda^{-(n+1)}=1$. Fix a point $w$ of the full shift $\{1,2, \dots, k\}^\mathbb{Z}$. We introduce the $(S,w)$-gap shift which is a generalization of the $S$-gap shift consisting of sequences in $\{0,1, \dots, k\}^\mathbb{Z}$ in which any two $0$'s are separated by a word $u$ appearing in $w$ such that $|u|\in S$. We extend the formula for the entropy of the $S$-gap shift to a formula describing the entropy of this new class of shift spaces. Additionally we investigate the dynamical properties including irreducibility and mixing of this generalization of the $S$-gap shift.
\end{abstract}

\bigskip

\section{Introduction}
Given an $S \subseteq \mathbb{Z}_{\geq 0}$, an $S$-gap shift, $X(S)$, is defined to be the shift space consisting of all sequences in $\{0,1\}^{\mathbb{Z}}$ such that any two nearest $1$'s are separated by a word $0^n$ for some $n \in S$.
The entropy of an $S$-gap shift is given by
$\log \lambda$, where $\lambda$ is the unique positive solution to the equation $\sum_{n \in S} \lambda^{-(n+1)}=1$. This formula for the entropy of an $S$-gap shift is given as an exercise in \cite[Exercise 4.3.7]{Lind-and-Marcus} and proofs of the formula are described in \cite[Corollary 3.21]{Garcia} and \cite{Blog}.

In \cite{matson}, Matson and Sattler generalize the notion of $S$-gap shifts by introducing the $\mathcal{S}$-limited shifts defined on an alphabet $\{1, \dots k\}$ by a collection of limiting sets  $\mathcal{S}=\{S_1, \dots, S_k\}$ where each $S_i \subseteq \N$ describes the allowed lengths of words in which the letter $i$ may appear.
They describe several dynamical properties of this class of shift spaces and prove a formula for the entropy of ordered $\mathcal{S}$-limited shifts. 
Dillon \cite{Dillon} further extends this 
defining the $\mathcal{S}$-graph shifts by a finite directed graph with a subset of the natural numbers assigned to each vertex.
The $\mathcal{S}$-graph shifts include all shifts of finite type and both the ordered and unordered $\mathcal{S}$-limited shifts and \cite{Dillon} computes the entropy of these shifts. These definitions generalize the $S$-gap shifts by altering the set $S$ of the allowed lengths of gaps between pairs of a symbol.

In this paper we extend the definition of $S$-gap shifts by modifying the words which separate pairs of symbols rather than the set $S$. 
Fix some $S \subseteq \Z_{\geq 0}$ and $w \in \{1,2, \dots, k\}^\Z$. 
We define an $(S,w)$-gap shift, denoted $X_w(S)$, to be 
the closure of the set of points of the form $\dots u_{-1} 0 u_0 0 u_1 \dots \in \{0,1, \dots, k\}^\Z$ where each $u_i$ is a word appearing in $w$ with $|u_i| \in S$.
As shown in section \ref{sec:S,w gap shifts}, any $(S,w)$-gap shift is irreducible and synchronized with synchronizing word 
$0$.
Additionally we show that an $(S,w)$-gap shift is mixing if and only if $\gcd\{n+1:n \in S\}=1$. 
It turns out that the entropy of an $(S,w)$-shift depends on the number of words of length $n$ appearing in $w$, we write this number $\phi_w(n)$. 
We apply Theorem \ref{thm:Garcia 3.12} (\cite[Corollary 3.12]{Garcia}) to prove the following formula for the entropy of $(S,w)$-gap shifts.
\begin{restatable}{thm}{EntropyFormula}\label{S,w Entropy}
Let $w \in \{1,2, \dots, k\}^\Z$ and $S \subseteq \Z_{\geq 0}$. Then $h(X_w(S))=\log \lambda$ where $\lambda>0$ is the unique positive solution to 
$$ 1=\sum_{n \in S}\phi_w(n)\lambda^{-(n+1)}.$$
\end{restatable}

\section{Background}

Let $\mathcal{A}$ be a finite set of symbols (or \emph{letters}), we refer to $\mathcal{A}$ as an \emph{alphabet}.
The collection of all bi-infinite sequences of symbols in $\mathcal{A}$ is the \emph{full shift} on $\mathcal{A}$ denoted 
$$\mathcal{A}^{\mathbb{Z}} = \{x = (x_{i})_{i\in \mathbb{Z}}\ : x_{i} \in \mathcal{A}\ \text{for all}\ i\in \mathbb{Z}\}$$
which is equivalent to the standard notation of $\mathcal{A}^\Z$ for the set of all maps $\Z \to \mathcal{A}$.

Each sequence $x=(x_i)_{i \in \mathbb{Z}}  \in \mathcal{A}^{\mathbb{Z}}$ is a \emph{point} of the full shift. A \emph{block} or \emph{word} over $\mathcal{A}$ is a finite sequence of symbols from the alphabet $\mathcal{A}$. The empty word containing no letters is denoted $\epsilon$. 
The \emph{length} of a word 
$u=u_1u_2 \cdots u_n \in \mathcal{A}^n$ is $|u|=n$ and the length of the empty word is $|\epsilon|=0$.
The restriction of $x \in \mathcal{A}^\Z$ to the set of integers in the interval $[i,j]$ with $i \leq j$ is
\begin{equation*}
    x_{[i,j]} = x_{i}x_{i+1}\ldots x_{j},
\end{equation*}
the block of coordinates in $x$ from position $i$ to position $j$. 
Similar for intervals $(i,j)$, $[i,j)$, $(i,j]$.
By extension we define 
$$x_{[i, \infty)}=x_ix_{i+1}\ldots ,$$
$$x_{(-\infty,i]}=\ldots x_{i-1} x_i.$$
A word 
$u$ \emph{appears} in the point $x=(x_i)_{i \in \Z}$ if
$u=x_{[i,j]}$ for some integers $i \leq j$.

Let $u \in \mathcal{A}^m$ and 
$v \in \mathcal{A}^n$
with $m \leq n$. 
If 
$v_{[0,m)}=u$ we say 
$u$ is a \emph{prefix} of 
$v$. If $v_{[n-m,n)}=u$ we say
$u$ is a \emph{suffix} of $v$.

The index $i$ of a bi-infinite sequence $x=(x_i)_{i \in \Z}$ may be thought of as the time. Then shifting the sequence one position to the left corresponds to moving forward in time.
\begin{defn}
    The \emph{shift map} $\sigma:\mathcal{A}^{\mathbb{Z}} \to \mathcal{A}^{\mathbb{Z}}$ maps a point $x$ to $y = \sigma(x)$ defined by $y_{i} = x_{i+1}$.
\end{defn}

\begin{defn}
    A point $x \in \mathcal{A}^\Z$ is said to be \emph{periodic for $\sigma$} if $\sigma^{p}(x)=x$ for some integer $p\geq 1$, in this case we say $x$ has \emph{period $p$ under $\sigma$}. 
    Equivalently, $x$ is a periodic point of period $p$ under the shift map $\sigma$ if $x=u^\infty=\ldots uuu \ldots$ for some word $u$ of length $p$.
    
    If $x$ is periodic, the smallest positive integer $p$ such that $\sigma^{p}(x)=x$ is the \emph{minimal period} of $x$. The point $x$ is a \emph{fixed point} for $\sigma$ if $\sigma(x)=x$.
\end{defn}

There are no restriction on what sequences can occur in the full shift  $\mathcal{A}^\Z$. Interesting dynamical properties arise when we impose constraints on which bi-infinite sequences occur in our space. Let $\mathcal{F}$ denote some set of words over the alphabet $\mathcal{A}$, the elements of $\mathcal{F}$ are known as \emph{forbidden words}. We define 
$$
X_\mathcal{F}=\{x \in \mathcal{A}^\Z: \text{ no word of }\mathcal{F} \text{ appears in }x\}.
$$
The notion of forbidden words leads us to the definition of a shift space.

\begin{defn}
A \emph{shift space} (or \emph{shift}) is a subset $X$ of the full shift $\mathcal{A}^\Z$ such that $X=X_\mathcal{F}$ for some collection $\mathcal{F}$ of forbidden words over $\mathcal{A}$.
\end{defn}

If there is some finite set of forbidden words $\mathcal{F}$ such that $X=X_\mathcal{F}$, we say $X$ is a \emph{shift of finite type}.
In many cases it's easier to describe what words are allowed to appear in a shift space rather than what words are forbidden.

\begin{defn}
    Let $X$ be a subset of a full shift $\mathcal{A}^{\mathbb{Z}}$.
    The set of all words of length $n$ that appear in some point of $X$ is denoted 
    $\mathcal{B}_{n}(X)$. 
    The \emph{language} of $X$ is the collection
    $$\mathcal{B}(X) = \bigcup_{n=0}^{\infty} \mathcal{B}_{n}(X). $$
    If $w$ is any point of $X$, we denote the set of all 
    words
    of length $n$ appearing in $w$ as $\mathcal{B}_{n}(w)$.
\end{defn}
The language of $X$ consists of all words which appear in some point of $X$.
By \cite[Proposition 1.3.4]{Lind-and-Marcus}, the language of a shift space determines the shift space.
For any subset $X \subseteq \mathcal{A}^{\mathbb{Z}}$, the condition that $x\in \mathcal{A}^{\mathbb{Z}}$ has each $x_{[i,j]} \in \mathcal{B}(X)$ for integers $i \leq j$ is equivalent to $X = X_{\mathcal{B}(X)^{c}}$. So $X \subseteq \mathcal{A}^{\mathbb{Z}}$ is a shift space if and only if whenever $x \in \mathcal{A}^\Z$ and each $x_{[i,j]} \in \mathcal{B}(X)$ then $x \in X$. Therefore, a shift space may be equivalently defined by its language.



For any shift space $X$, the \emph{cylinder set} of
$u \in \mathcal{B}_n(X)$
is defined as
$$ [u]=\{x\in X: x_{[0,n)}=u\}.$$
Take the discrete product topology on the full shift $\mathcal{A}^\Z$.
The collection of cylinder sets forms a basis for the topology on $X \subseteq \mathcal{A}^\Z$.
The \emph{extender set} of
$u \in \mathcal{B}_n(X)$
is
$$E_X(u)=\{x_{(- \infty,0)}x_{[n,\infty)}: x \in [u]\}.$$

\begin{example}
The extender set of any word $u$ in the language of $X=\{0,1\}^\Z$ is $E_X(u)=\{0,1\}^\Z$.
\end{example}

Now we discuss several important dynamical properties that a shift space may have. 
\begin{defn}
Let $X$ be a shift space. A word 
$v \in \mathcal{B}(X)$ is \emph{synchronizing} if 
$uv, vw \in \mathcal{B}(X)$ implies 
$uvw \in \mathcal{B}(X)$. If $\mathcal{B}(X)$ contains a synchronizing word, we say $X$ is \emph{synchronized}.
\end{defn}

\begin{defn}
A shift space $X$ is said to be \emph{irreducible} if for all ordered pairs 
$(u,w)$ with $u,w \in \mathcal{B}(X)$ there is some 
$v \in \mathcal{B}(X)$ such that 
$uvw \in \mathcal{B}(X)$.
\end{defn}

\begin{defn}
A shift space $X$ is \emph{mixing} if for all 
$u,w \in \mathcal{B}(X)$, there is some positive integer $N$ such that for each $n \geq N$ there exists a word 
$v \in \mathcal{B}_n(X)$ such that
$uvw \in \mathcal{B}(X)$.
\end{defn}

\begin{defn}
    The \emph{(topological) entropy} of a shift space $X$ is defined as
    $$h(X)=h_{\text{top}}(X) = \lim_{n \rightarrow \infty} \frac{1}{n} \log{|\mathcal{B}_{n}(X)|}.$$
\end{defn}

For the remainder of this paper any measure $\mu$ will be assumed to be a Borel probability measure on a shift space $X$. Additionally, we assume that $\mu$ is invariant under the shift map $\sigma$, that is, $\mu(\sigma^{-1}(B))=\mu(B)$ for any measurable set $B$. We denote the set of all such measures on $X$ as $\mathcal{M}(X)$.

\begin{defn}
The \emph{metric entropy} or \emph{measure-theoretic entropy} of $X$ with a given measure $\mu$ is defined as
$$
h_\mu(X)=\lim_{n \to \infty} \frac{1}{n} \sum_{\alpha \in \mathcal{A}^n}-\mu[\alpha]\log (\mu[\alpha]).
$$
\end{defn}
Since any word $\alpha= \alpha_1 \dots \alpha_n \in \mathcal{A}^{n+m}$ can be broken into two words $\alpha_{[1,n]} \in \mathcal{A}^n$, $\alpha_{[n+1,n+m]} \in \mathcal{A}^m$, the sequence $\sum_{\alpha \in \mathcal{A}^n}-\mu[\alpha]\log (\mu[\alpha])$ is subadditive in $n$. By Fekete’s Lemma, the limit defining the measure-theoretic entropy exists and we have
$$h_\mu(X)=\lim_{n \to \infty} \frac{1}{n} \sum_{\alpha \in \mathcal{A}^n}-\mu[\alpha]\log (\mu[\alpha])=\inf_{n \geq 1} \frac{1}{n} \sum_{\alpha \in \mathcal{A}^n}-\mu[\alpha]\log (\mu[\alpha]).$$

The
topological entropy and metric entropy 
of a shift $X$ are related by the \emph{variational principle} (see, e.g. \cite[Theorem 8.6]{Walters}) which states:
$$
h(X)= \sup_{\mu \in \mathcal{M}(X)} h_\mu(X)
$$
A measure $\mu$ is called a \emph{measure of maximal entropy}, or MME, if $h_{\mu}(X)=h(X)$.

\begin{remark}\label{MME exists for shift space}
Any shift space $X$ is a compact metric space with the metric 
$$d_\theta(x,y)=\theta^{\max\{k:x_i=y_i, |i|\leq k\}}$$
for $0<\theta<1$ and 
the shift map $\sigma:X \to X$ is an expansive homeomorphism. By \cite[Theorem 8.2]{Walters},
the entropy map $\mu \mapsto h_\mu(X)$ is upper semi-continuous. 
Then the entropy map attains a maximum $h_\mu(X)$ on $\mathcal{M}(X)$ and by the variational principle, $h(X)=h_\mu(X)$. Hence, a measure of maximal entropy exists on any shift space $X$.
\end{remark}

\section{$S$-Gap Shifts}

\begin{defn}\label{S-gap defn}
    Let $S \subseteq \Z_{\geq 0}$. The \emph{$S$-gap shift} 
    is 
    $$X_S=\overline{\{\dots 10^{n_{-1}}10^{n_0}10^{n_1} \dots: n_i \in S\}} \subseteq \{0,1\}^\Z. $$
    Equivalently, an $S$-gap shift is the shift space $X_w(S)$
    on the alphabet $\{0,1\}$ 
    consisting of points of the form 
$$\dots 10^{n_{-1}}10^{n_0}10^{n_1} \dots$$
    with $n_i \in S$. In the case that $S$ is infinite, we also allow sequences which begin and/or end with an infinite string of $0$'s.
\end{defn}

\begin{thm}[{{\cite[p. 1407]{Jung}}}]
    An $S$-gap shift is mixing if and only if $$\gcd\{n+1:n \in S\}=1.$$
\end{thm}

A formula for the entropy of the $S$-gap shift is 
given as an exercise in \cite[Exercise 4.3.7]{Lind-and-Marcus} and is computed in
\cite[Corollary 3.21]{Garcia}.
Several other proofs of this formula are described in \cite{Blog} one of which is generalized to the ordered $\mathcal{S}$-limited shifts in \cite{matson}. Additionally, we 
present a proof of the $S$-gap entropy formula, Theorem \ref{thm:Entropy of S-gap}.
This proof has the advantage of taking a combinatorial approach, but does not generalize well for the $(S,w)$-gap shifts. So we adopt an approach similar to that of \cite[Corollary 3.21]{Garcia} for our proof of Theorem \ref{S,w Entropy}.

\begin{thm}\label{thm:Entropy of S-gap}
    The \emph{entropy} of any $S$-gap shift is 
    $h(X(S)) = \log \lambda$
    where $\lambda>0$ is the unique positive solution to 
    $$ 1=\sum_{n \in S}\lambda^{-(n+1)}.$$
\end{thm}

\begin{proof}
Let $a_n=|\mathcal{B}_n(X(S))|$ and consider the function 
$$H(z)=\sum_{n=0}^\infty a_n z^n.$$
We compute that the radius of convergence of $H(z)$ is $\xi=e^{-h(X(S))}$ as follows 
\begin{align*}
    \log (\xi) &= \log (\lim_{n \to \infty} \sqrt[n]{a_n}^{-1})\\
    &= \lim_{n\to \infty} -\log(\sqrt[n]{a_n}) \\
    &= -\lim_{n \to \infty} \frac{1}{n} \log(a_n)\\
    &=-h(X(S)).
\end{align*}
Put $G=\{0^n1:n \in S\}$ and let $A_n^k$ be the set of all words of length $n$ that may be written as the concatenation of exactly $k\geq 1$ words in $G$, $A_n^0:=\{\epsilon\}$.
Define 
$$F_k(z)=\sum_{n=1}^\infty |A_n^k|z^n.$$
Note that in the case of $k=1$,
$$F_1(z)=\sum_{n=1}^\infty |A_n^1|z^n=\sum_{n \in S}z^{n+1}$$
because $|A_n^1|=\mathbf{1}_S(n-1)$.

We claim that $F_k(z) F_\ell(z)=F_{k+\ell}(z)$.
It is sufficient to show that 
$$\sum_{m=1}^n |A_{n-m}^k||A_m^\ell|=|A_n^{k+\ell}|.$$
First consider words $u \in A_m^\ell, v \in A_{n-m}^k$ for any $1 \leq m \leq n$.
Since $uv$ has length $n$ and can be written as the concatenation of $\ell$ words in $G$ followed by $k$ words in $G$, we have $uv \in A_n^{k+\ell}$ and thus
$$
\sum_{m=1}^n|A_{n-m}||A_m^\ell| \leq |A_n^{k+\ell}|.
$$
Now consider any $w \in A_n^{k+\ell}$ and break the word after the $\ell$-th $2$ denoting the first part as $u$ and the last part as $v$. Then $uv=w$, $u\in A_m^\ell, v \in A_{n-m}^k$ where $m=|u|$. So 
$$
|A_n^{k+\ell}| \leq \sum_{m=1}^n|A_{n-m}||A_m^\ell|.
$$
By induction, it follows that $F_k(z)=(F_1(z))^k$.

Since $A_n^k$, $k \geq 1$ are disjoint subsets of $\mathcal{B}_n(X(S))$, we have $$\sum_{k \geq 1} |A_n^k| \leq |\mathcal{B}_n(X(S))|.$$
Any word of $\mathcal{B}_n(X(S))$ is of the form $0^n$ or $0^i 1w0^j$ where $w \in A_{n-i-j-1}^k$ for some $k \geq 0$ so 
$$|\mathcal{B}_n(X(S))| \leq 1+\sum_{k\geq 0}\sum_{i \geq 0} \sum_{j \geq 0}|A_{n-i-j-1}^k|.$$
If $z \geq 0$, then we may apply Tonelli's theorem and compute
\begin{gather*}
\sum_{k \geq 1}|A_n^k| \leq |\mathcal{B}_n(X(S))| \leq 1+\sum_{k\geq 0}\sum_{i \geq 0} \sum_{j \geq 0}|A_{n-i-j-1}^k|, \\
\sum_{n\geq 1}\sum_{k \geq 1}|A_n^k|z^n \leq H(z) \leq \sum_{n \geq 1}z^n+\sum_{n \geq 1}\sum_{k\geq 0}\sum_{i \geq 0} \sum_{j \geq 0}|A_{n-i-j-1}^k|z^n,  \\
\sum_{k \geq 1}F_k(z) \leq H(z) \leq \sum_{n \geq 1}z^n+ 
\sum_{k \geq 0} F_k(z) \left(\sum_{i \geq 0}z^i\right) \left(\sum_{j \geq 0} z^{j+1} \right)
\end{gather*}
where the last inequality is because
\begin{align*}
\sum_{n \geq 1}\sum_{k\geq 0}\sum_{i \geq 0} \sum_{j \geq 0}|A_{n-i-j-1}^k|z^n
    &=\sum_{n \geq 1}\sum_{k\geq 0}\sum_{i \geq 0} \sum_{j \geq 0}|A_n^k|z^{n+i+j+1}\\
    &=\sum_{n \geq 1}\sum_{k\geq 0}|A_n^k|z^n\left(\sum_{i \geq 0} \sum_{j \geq 0} z^{i+j+1}\right)\\
    &=\sum_{k \geq 0}F_k(z) \left(\sum_{i \geq 0}z^i\right) \left(\sum_{j \geq 0} z^{j+1} \right).
\end{align*}

If $\hat{H}(z)=\sum_{k \geq 1}F_k(z)=\sum_{k \geq 1} (F_1(z))^k$, then
$$
\hat{H}(z) \leq H(z) \leq \sum_{n \geq 1}z^n+(F_0(z)+\hat{H}(z))\left(\sum_{i \geq 0}z^i\right) \left(\sum_{j \geq 0} z^{j+1} \right).
$$
Since $\sum_{n \geq 1}z^n$ converges for all $z \in [0,1)$ and $F_0(z)= \sum_{n \geq 1}z^n$, $\hat{H}(z)$ converges if and only 
if
$H(z)$ converges. Then the radius of convergence of $H(z)$ is the unique $z>0$ such that $F_1(z)=1$ and by the first paragraph $h(X(S))=\log(1/z).$
\end{proof}

\section{$(S,w)$-gap Shifts}
\label{sec:S,w gap shifts}

\begin{defn}
Let $S \subseteq \Z_{\geq 0}$ and let $w \in \{1,2, \dots, k\}^\Z$ 
where $k>0$ is an integer.
We define the \emph{$(S,w)$-gap shift} 
to be 

$$X_w(S)=\overline{
\{\ldots 0 u_{-1} 0 u_0 0 u_1 \ldots: u_i \text{ appears in }w, |u_i| \in S\}}\subseteq \{0,1, \dots, k\}^\Z.$$

\end{defn}

Recall that $\mathcal{B}_n(w)$ is the set of all length $n$ words that appear in $w$. Let $B=\mathcal{B}(\{1,2, \dots, k\}^\mathbb{Z}) \setminus \left( \bigcup_n \mathcal{B}_n(w)\right)$. If $S$ is infinite, then the forbidden word set of $X_w(S)$ is
\begin{equation}\label{forbidden words for S,w-gap (S infinite)}
\mathcal{F}=
B
\cup \{0u: u \in B\}
\cup \{u0: u \in B\}
\cup \{0u0:|u| \in \Z_{\geq 0} \setminus S\} 
.
\end{equation}
If $S$ is finite, then the forbidden word set of $X_w(S)$ is $$\mathcal{F} \cup \{1,2, \dots, k\}^{1+\max S}.$$ So the $(S,w)$-gap shift is in fact a shift space.

\begin{example}
Let $S \subseteq \Z_{\geq 0}$. By relabeling symbols of the alphabet $\{0,1\}$ in Definition \ref{S-gap defn}, we note that an alternative definition of the $S$-gap shift $X(S)$ is the set of points of $\{0,1\}^\Z$ such that any pair of nearest $0$'s is separated by a block $1^n$ with $n \in S$.
Therefore, any $S$-gap shift is the $(S,w)$-gap shift with $w=1^\infty$.
\end{example}

\begin{example}
If $S\subseteq \Z_{\geq 0}$ is cofinite and words in the language of $\{1,\dots, k\}^\Z$
appear in $w$, then the set described in (\ref{forbidden words for S,w-gap (S infinite)}) is finite and thus $X_w(S)$ is a shift of finite type. For example we may take a point $w \in \{1,\dots, k\}^\Z$ such that $w_i=1$ for $i < 0$ and $w_{[0,\infty)}=\alpha_0 \alpha_1 \ldots \alpha_n \ldots$ 
where $\alpha_n$ denotes the concatenation of all elements of $\{1, \dots, k\}^n$.
\end{example}

\begin{prop}
Any $(S,w)$-gap shift is synchronized with synchronizing word $0$.
\end{prop}
\begin{proof}
If $u0 \in \mathcal{B}(X_w(S))$ and $ 0v \in \mathcal{B}(X_w(S))$, then $ u0v \in \mathcal{B}(X_w(S))$. So $ 0$ is a synchronizing word for $X_w(S)$.
\end{proof}

\begin{prop}\label{S,w-gap irreducible}
Any $(S,w)$-gap shift is irreducible.
\end{prop}
\begin{proof}
Let 
$u,v \in \mathcal{B}(X_w(S))$.
To show $X_w(S)$ is irreducible, we will construct a word $t0s$ such that $ut0sv \in \mathcal{B}(X_w(S))$.

If 
$u$ ends with 
$0$, put $t=w_{[0,n)}$
for some $n\in S$. Otherwise, $u$
has a suffix of the form 
$w_{[i,i+k)}$ with $k$ less than or equal to some element $n \in S$. If $k \in S$, put $t= \epsilon$. If $k \notin S$, put $t=w_{[i+k,i+n)}$. Then $ut0 \in \mathcal{B}(X_w(S))$.
Similarly, if $v$
begins with $ 0$,
put $t=w_{[0,n)}$
for some $n \in S$. Otherwise, 
has a prefix of the form 
$w_{[i-k,i)}$ with $k$ less than or equal to some element $n \in S$.
If $k \in S$, put $s=\epsilon$. If $k \notin S$, put $s=w_{[i-n,i-k)}$. Then $0sv \in \mathcal{B}(X_w(S))$.

Since $0$ is a synchronizing word for $X_w(S)$, $ut0sv \in L(X_w(S))$. Hence $X_w(S)$ is irreducible.
\end{proof}

\begin{remark}\label{irreducible and separated by a 2}
Note that 
the construction 
in the proof of Proposition \ref{S,w-gap irreducible} 
shows that for all 
$u,v \in \mathcal{B}(X_w(S))$
there are words $t,s \in \mathcal{B}(X_w(S))$
such that 
$ut0sv \in \mathcal{B}(X_w(S))$. 
\end{remark}

The following proposition uses this observation and is a slight variant of the proof for $S$-gap shifts in \cite[p. 1407]{Jung} and for the $\mathcal{S}$-graph shifts in \cite[Proposition 3.4.]{Dillon}.
\begin{prop}
An $(S,w)$-gap shift is mixing if and only if $$\gcd\{n+1:n \in S\}=1.$$
\end{prop}
\begin{proof}
If $X_w(S)$ is mixing, then there is some $N>0$ such that for all $n\geq N$ there is some word $\gamma \in \mathcal{B}_n(X_w(S))$ such that 
$0 \gamma 
0 \in \mathcal{B}(X_w(S))$. In particular, we may choose $\gamma_1\in \mathcal{B}_N(X_w(S))$,$\gamma_2 \in \mathcal{B}_{N+1}(X_w(S))$ such that 
$0 \gamma_1 0,
0 \gamma_2 0 \in \mathcal{B}(X_w(S))$ and $\gamma_1, \gamma_2$ are of the form 
$$u_{1} 0 u_{2} \cdots  0 u_{m}$$
for $u_{i} \in \bigcup_{k \in S} \mathcal{B}_k(w)$. Then 
$0\gamma_1 \in \mathcal{B}_{N+1}(X_w(S))$ and 
$0 \gamma_2 \in \mathcal{B}_{N+2}(X_w(S))$ are concatenations of words of the form 
$0 u_{i}$. Then $N+1$ and $N+2$ are equal to sums of elements of $\{n+1: n \in S\}$ because 
$|0 u_{i}|=|u_{i}|+1 \in \{n+1:n \in S\}$. Since $N+1$ and $N+2$ are relatively prime, $\gcd\{n+1:n \in S\}=1$.

Now suppose that $\gcd \{n+1:n \in S\}=1$ and let $\alpha, \beta \in \mathcal{B}(X_w(S))$.
Then there is some $M>0$ such that any integer $n\geq M+1$ may be written as a sum of elements of $\{n+1:n \in S\}$ and thus there is a word 
$$0 v_1 0 v_2 \cdots 0 v_{m'} 0$$
of length $n$ with $v_i\in \bigcup_{k \in S} \mathcal{B}_k(w)$.
As noted in remark \ref{irreducible and separated by a 2}, there is a word 
$\gamma_1 0 \gamma_2 \in \mathcal{B}(X_w(S))$ such that $\alpha \gamma_1 0 \gamma_2\beta \in \mathcal{B}(X_w(S))$. 
Let $N=M+|\gamma_1\gamma_2|+1$. For any $n\geq N$ there is a word
$$0 u_1 0 u_2 \cdots 0 u_{m} 0$$
of length $n-|\gamma_1\gamma_2|\geq M+1$ with $u_i\in \bigcup_{k \in S} \mathcal{B}_k(w)$.
Since $0$ is a synchronizing word for $X_w(S)$, 
$$\gamma:=\gamma_1 0 u_1 0 u_2 \cdots 0 u_{m} 0 \gamma_2 \in \mathcal{B}(X_w(S))$$ is a word of length $n$ such that $\alpha \gamma \beta \in \mathcal{B}(X_w(S))$. This shows $X_w(S)$ is mixing.
\end{proof}

\begin{defn}
A shift space $X$ is a \emph{coded system} if it is the closure of the set of points obtained by freely concatenating elements of some set $G$ of words over an alphabet $\mathcal{A}$. The elements of $G$ are called \emph{generators} of the coded system.
\end{defn}

The $(S,w)$-gap shifts are an example of a coded system with the generating set $$G=\bigcup_{n\in S} \bigcup_{u \in B_n(w)} u 0 =\{ u 0 : u \text{ appears in } w, |u|\in S\}.$$

\section{Entropy}

\begin{lemma}[{{\cite[
Lemma 3.14]{Garcia}}}]\label{lemma:extender sets} 
If $v$ is a synchronizing word for a shift space $X$, then for any 
word of the form $v u v \in \mathcal{B}(X)$, we have $$E_X(v u v)=E_X(v).$$
\end{lemma}

\begin{thm}[{{\cite[
Corollary 3.12]{Garcia}}}]\label{thm:Garcia 3.12}
    Let $X$ be a subshift, $\mu$ a measure of maximal entropy and $u,v \in \mathcal{B}(X)$. If $E_{X}(u)=E_{X}(v)$, 
    then 
    $$\mu[v]=\mu[u]e^{h(X)(|u|-|v|)}.$$
\end{thm}

\begin{lemma}\label{MME with mu[0]>0}
For any $w \in \{1,2, \dots,k\}^\Z$
and $S \subseteq \Z_{\geq 0}$ there is some measure of maximal entropy $\mu$ on $X_w(S)$ such that $\mu[0]>0$.
\end{lemma}
\begin{proof}
By Remark \ref{MME exists for shift space}, there exists a measure of maximal entropy on $X_w(S)$. If $S$ is finite, then any pair of nearest $0$'s appearing in a point of $X_w(S)$ are separated by a word $u$ appearing in $w$ with $|u|\leq \max S$. 
Then the sets $[0], \sigma[0], \dots, \sigma^{\max S}[0]$ cover $X_w(S)$, so any measure of maximal entropy $\mu$ on $X_w(S)$ satisfies $\mu[0] \geq (1+\max S)^{-1}>0$.

Now consider the case when $S$ is infinite. 
Suppose that $\mu$ is measure of maximal entropy on $X_w(S)$ such that $\mu[0]=0$. Then $\mu$ is supported on the orbit closure of $w$, that is, the shift space $\mathcal{O}$ defined by the language $\mathcal{B}(w)=\bigcup_n \mathcal{B}_n(w)$. 
By the variational principle, $h(\mathcal{O}) \geq h_\mu(\mathcal{O})=h(X_w(S))$. 
Since $\log(t!)/t \to \infty$ as $t \to \infty$, we may choose some positive integer $t$ such that $\log(t!)/t >h(\mathcal{O})$. Since $S$ is infinite, we there are distinct $n_1, \dots, n_t \in \Z_{\geq 0}$ such that $n_1-1, \dots, n_t-1 \in S$. For any $k>0$ the set $\mathcal{B}_{k(n_1+n_2 + \cdots +n_t)}(X_w(S))$ contains all word of the form 
$$(0u_{1,1}0u_{1,2}\dots 0u_{1,t})(0u_{2,1}0u_{2,2}\dots 0u_{2,t})\dots (0u_{k,1}0u_{k,2}\dots 0u_{k,t})$$
where the lengths $|u_{i,1}|,|u_{i,2}|, \dots, |u_{i,t}|$ are distinct elements of $\{n_1-1,n_2-1, \dots, n_t-1\}$ for each $i=1, \dots, k$. So 
$$|\mathcal{B}_{k(n_1+n_2 + \cdots +n_t)}(X_w(S))| \geq (t!)^k |\mathcal{B}_{n_1-1}(w)|^k|\mathcal{B}_{n_2-1}(w)|^k \cdots |\mathcal{B}_{n_t-1}(w)|^k.$$
Since $|\mathcal{B}_n(w)| \geq e^{n h(\mathcal{O})}$ for all values of $n$, we have
\begin{align*}
    \frac{\log |\mathcal{B}_{k(n_1+n_2 + \cdots +n_t)}(X_w(S))|}{k(n_1+n_2 + \cdots +n_t)} 
    &\geq \frac{\log((t!)^k |\mathcal{B}_{n_1-1}(w)|^k \cdots |\mathcal{B}_{n_t-1}(w)|^k)}{k(n_1+n_2 + \cdots +n_t)}
\\
& \geq \frac{k \log(t!)+k(n_1+n_2+ \cdots+ n_t-t)h(\mathcal{O})}{k(n_1+n_2 + \cdots +n_t)}\\
&=\frac{\log(t!)-t\cdot h(\mathcal{O})+(n_1+n_2+ \cdots n_t) h(\mathcal{O})}{n_1+n_2 + \cdots +n_t}.
    \end{align*}
By taking the limit $k \to \infty$,  
$$
h(X_w(S)) 
=\frac{\log(t!)-t\cdot h(\mathcal{O})+(n_1+n_2+ \cdots n_t) h(\mathcal{O})}{n_1+n_2 + \cdots +n_t} 
>h(\mathcal{O})
$$
since $\log(t!)>t \cdot h(\mathcal{O})$.
This is a contradiction to $h(\mathcal{O}) \geq h(X_w(S))$. 
\end{proof}

For any point $w \in \{1,2, \dots, k\}^\Z$
define $\phi_w(n)$ for $n \geq 0$ to be the number of distinct words of length $n$ appearing in $w$. We say $\phi_w$ is the \emph{complexity function of $w$}. Observe that $\phi_w(n)=|\mathcal{B}_n(w)|$.

\EntropyFormula*

\begin{proof}
In any $(S,w)$-gap shift 
$X_w(S)$, we have
\begin{equation}\label{cylinder set of 0}
[ 0 ]=\left(\bigcup_{n \in S} \bigcup_{u \in \mathcal{B}_n(w)} [0 u 0 ]\right) \cup \left(\bigcup_{i \in \Z} [ 0 w_{[i,\infty)}]\right)
\end{equation}
where $[ 0 w_{[i, \infty)}]$ denotes the set of all $x \in X_w(S)$ such that $x_{[0,\infty)}= 0 {w_{[i, \infty)}}$. 
By Lemma \ref{MME with mu[0]>0}, there is some measure of maximal entropy $\mu$ on $X_w(S)$ such that $\mu[0]>0$.
Suppose that for some $i \in \Z$, $\mu[
0 w_{[i,\infty)}] \neq 0$. Since $w \in 
\{1,2, \dots, k\}^\Z$, all $\sigma^{-n}[0
w_{[i,\infty)}]$
for $n \geq 0$ are pairwise disjoint. By the property that $\sigma$ is measure preserving, 
$$
\mu \left(\bigcup_{n \geq 0} \sigma^{-n}[0w_{[i,\infty)}] \right)
=\sum_{n \geq 0} \mu(\sigma^{-n}[0w_{[i,\infty)}])=\sum_{n\geq 0} \mu[0w_{[i,\infty)}],
$$
a contradiction to the fact that $\mu(X_w(S))=1$.
Hence, $\mu[
0
w_{[i,\infty)}]=0$ for all $i \in \Z$ and we have $\mu\left(\bigcup_{i \in \Z} [0w_{[i,\infty)}]\right)=0.$

Since $0$ is a synchronizing word for $X_w(S)$, Lemma \ref{lemma:extender sets} shows that 
$$E_{X_w(S)}(0u0
)=E_{X_w(S)}(0)$$
for all $u \in \mathcal{B}_n(w)$. By
Theorem
\ref{thm:Garcia 3.12}, 
$$\mu[0u0]=\mu[0]e^{h(X_w(S))(-n-1)}.$$
For any pair of words $u \neq v$ appearing in $w$, the cylinder sets $[0u0]$ and $[0v0]$ are disjoint.
By taking the measure of the set in (\ref{cylinder set of 0}),
\begin{align*}
\mu[0]&=\mu \left(\bigcup_{n \in S} \bigcup_{u \in \mathcal{B}_n(w)} [
0u0]\right)  \\
&= \sum_{n \in S} \sum_{u \in \mathcal{B}_n(w)} \mu[0 u0]\\
&= \sum_{n \in S} \phi_w(n) \mu[0]e^{h(X_w(S))(-n-1)}.
\end{align*}
Dividing by $\mu[0]>0$ shows
$$
1=\sum_{n \in S} \phi_w(n) e^{h(X_w(S))(-n-1)}.
$$
Since the complexity function of $w$ is positive, the positive solution to 
$$
1=\sum_{n \in S}\phi_w(n) \lambda^{-(n+1)}
$$
is unique.
\end{proof}

\begin{corollary}
If $w \in 
\{1,2, \dots, k\}^\Z$
has minimal period $p$ and $S \subseteq \Z_{\geq 0}$, then $h(X_w(S))=\log \lambda$ where $\lambda>0$ is the unique solution to 
$$
1= \sum_{n \in S, n <p-1} \phi_w(n) \lambda^{-(n+1)}+ p \sum_{n \in S, n\geq p-1} \lambda^{-(n+1)}.
$$
\end{corollary}

\begin{proof}
It suffices to show that 
$\varphi_{w}(n)=p$ for $n> p-1$. Let $w\in
\{1,2, \dots, k\}^\Z$
have minimal period $p$. 
Consider the case when $n=p-1$. 
Showing that the elements in the set 
$$\mathcal{T} = \{\sigma^{m}_{[0,n-1]}(w)\ |\ 0\leq m \leq p-1\}$$
are distinct is sufficient enough to complete the proof. Here we denote $\sigma^{m}_{[0,n-1]}(w)$ as the shift map on the cylinder set $w$ from the coordinates $0$ to $n-1$. Now suppose, for contradiction, that the elements in the set $\mathcal{T}$ were not all distinct. Since $w$ has period $p$ and if we know the first $p-1$ symbols, then we know what the last symbol must be to satisfy that $w$ has period $p$. This allows us to write
$$\sigma^{i}(w) = \sigma^{j}(w)$$
for some $0\leq i < j \leq p-1$. Taking $\sigma^{-i}$ on both sides yields
$$w = \sigma^{-i}(\sigma^{j}(w))=\sigma^{j-i}(w)$$
which contradicts the fact that $w$ has period $p$. Now consider the case for when $n>p-1$. Once again, we assume the elements in $\mathcal{T}$ are not distinct. We have the following equality
$$\sigma_{[0,n-1]}^{i}(w) = \sigma_{[0,n-1]}^{j}(w)$$
for some $0\leq i < j \leq p-1$. The above equality states that the two words agree with each other on the the first $n$ symbols, so they must agree everywhere with their periodicity infinitely in both directions. This allows us to write the equality as $\sigma^{i}(w) = \sigma^{j}(w)$. We achieve the same contradiction with a similar argument. So, $\varphi_{w}(n)=p$ for $n> p-1$.
\end{proof}

\section*{Acknowledgments} 
We thank Daniel Thompson for all of the guidance throughout this project 
which was completed as part of the Ohio State University's ROMUS program.
We also thank Austin Allen, Katelynn Huneycutt, Michael Lane, Thomas O'Hare, and Samantha Sandberg for all of their helpful comments and Ronnie Pavlov for all of his suggestions and feedback.
We thank the referee for the very thorough reading of our paper and for their comments which have helped improve the paper.

{
}

\end{document}